\documentclass[11pt]{article}
\usepackage{amsmath,amssymb,amsthm}
\usepackage{graphicx,subfigure,float,url}
\usepackage{pdfsync}
\usepackage[english]{babel}
\usepackage[utf8]{inputenc}
\usepackage{tipa}
\usepackage{tikz}

\newlength\imagewidth
\newlength\imagescale

\usepackage{fancybox}
\usepackage{listings}
\usepackage{mathtools}
\usepackage{amsfonts}
\usepackage[bottom]{footmisc}
\usepackage{verbatim}

\usepackage{colortbl}

\usepackage{float}
\usepackage{makeidx}

\normalbaselines

\newtheorem{Theorem}{Theorem}[section]

\newtheorem{Definition}[Theorem]{Definition}

\newtheorem{Lemma}[Theorem]{Lemma}

\newtheorem{Corollary}[Theorem]{Corollary}

\newtheorem{Remark}[Theorem]{Remark}

\newcommand{\RR}{{{\rm I} \kern -.15em {\rm R} }}

\newcommand{\C}{{{\rm l} \kern -.42em {\rm C} }}

\newcommand{\nat}{{{\rm I} \kern -.15em {\rm N} }}

\newcommand{\dv}{\mbox{\rm div}\ }

\newcommand{\be}{\begin{equation}}
\newcommand{\ee}{\end{equation}}
\newcommand{\beq}{\begin{eqnarray}}
\newcommand{\eeq}{\end{eqnarray}}
\newcommand{\beqs}{\begin{eqnarray*}}
\newcommand{\eeqs}{\end{eqnarray*}}
\newcommand{\bt}{\begin{Theorem}}
\newcommand{\et}{\end{Theorem}}
\newcommand{\br}{\begin{Remark}}
\newcommand{\er}{\end{Remark}}
\newcommand{\bc}{\begin{Corollary}}
\newcommand{\ec}{\end{Corollary}}
\newcommand{\bl}{\begin{Lemma}}
\newcommand{\el}{\end{Lemma}}
\newcommand{\bd}{\begin{definition}}
\newcommand{\ed}{\end{definition}}

\renewcommand{\geq}{\geqslant}
\renewcommand{\leq}{\leqslant}

\graphicspath{{figures/}}

\title{Convergence to consensus for a Hegselmann-Krause-type model with distributed time delay}
\author{
Alessandro Paolucci\footnote{Dipartimento di Ingegneria e Scienze dell'Informazione e Matematica, Universit\`{a} di L'Aquila, Via Vetoio, Loc. Coppito, 67010 L'Aquila Italy (\texttt{alessandro.paolucci2@graduate.univaq.it}).}
}

\date{}

\begin{document}

\textwidth=160 mm

\textheight=225mm

\parindent=8mm

\frenchspacing

\maketitle

\begin{abstract}
In this paper we study a Hegselmann-Krause opinion formation model with distributed time delay and positive influence functions. Through a Lyapunov functional approach, we provide a consensus result under a smallness assumption on the initial delay. Furthermore, we analyze a transport equation, obtained as mean-field limit of the particle one. We prove global existence and uniqueness of the measure-valued solution for the delayed transport equation and its convergence to consensus under a smallness assumption on the delay, using a priori estimates which are uniform with respect to the number of agents.
\end{abstract}

\vspace{5 mm}

\def\qed{\hbox{\hskip 6pt\vrule width6pt
height7pt
depth1pt  \hskip1pt}\bigskip}


\section{Introduction}
In recent years, many researchers have focused their attention to multi-agent systems. One aspect of these models is the natural self-aggregation, which has been studied in different fields such as biology \cite{Cama}, robotics \cite{Desai}, sociology, economics \cite{marsan}, computer science, control theory \cite{Piccoli, PRT, WCB}, social sciences \cite{Campi, T} and many other areas. In these last decades a large number of mathematical models has been proposed to study the consensus behavior. First order models, such as the Hegselmann-Krause model \cite{HK}, have been proposed to study opinion formation. We mention also \cite{JM}, in which bounded confidence yields the so-called clustering phenomenon. Second order models, in particular Cucker-Smale model \cite{CS}, have been studied by many authors \cite{HL, Ha e Tadmor, PR}, in order to describe, for example, flocking of birds, swarming of bacteria, or schooling of fishes.

In addition, it is reasonable to introduce a delay in the model as a reaction time or simply as a time to receive the information from outside, in order to let the dynamics more realistic. For first order models, we refer to \cite{CF, CPP, CP2}, while for delayed Cucker-Smale-type models we mention \cite{CH, CL, HM, PT}. In particular, in very recent papers (see \cite{CP, Liu, PR2}), the authors analyzed modified Cucker-Smale models with distributed time delay, thanks to which agents are influenced by the other ones on a time interval $[t-\tau(t),t].$

Furthermore, delayed and non-delayed kinetic and transport equations associated to the particle multi-agent systems have been studied in \cite{CCR, CFT, CFT2, CH, CPP, CP}.

In this paper, we are interested in the evolution of opinions among $N$ agents, with $N\in \nat$. Let $x_i\in \RR^d$ be the opinion of the $i$-th agent, for any $i=1,...,N$. Then, the dynamics is given by the following Hegselmann-Krause-type model:
\begin{equation}
\label{modello}
\begin{array}{l}
\displaystyle{\frac{dx_i(t)}{dt}=\frac{1}{Nh(t)}\sum_{j\neq i} \int_{t-\tau(t)}^t \alpha (t-s) a_{ij}(t;s) (x_j(s)-x_i(t)) ds, \quad t>0}\\
\hspace{0.21 cm}
\displaystyle{ x_i(s)=x_{i,0}(s), \quad s\in [-\tau(0),0],}
\end{array}
\end{equation}
where $\tau:[0,+\infty)\to(0,+\infty)$ is the time delay. It is a function in $W^{1,\infty}([0,T])$, for any $T>0$ and we assume that $\tau(t)\geq \tau_*$ for some $\tau_*>0$, and
\begin{equation}
\label{ritardo1}
\tau'(t)\leq 0, \qquad \forall \ t\geq 0.
\end{equation}
This implies that $\tau(t)\leq \tau(0)$, for any $t\geq 0$. We stress the fact that constant delays $\tau(t)\equiv \bar{\tau}>0$ are allowed.  

Motivated by \cite{CS, JM, MT}, we take the communication rates $a_{ij}(t;s)$ either of the form
\begin{equation}
\label{symmetric}
a_{ij}(t;s)=\psi(|x_j(s)-x_i(t)|), 
\end{equation}
for any $i,j\in \{ 1,\dots ,N\}$, where $\psi:[0,+\infty)\rightarrow (0,+\infty)$ is a non-increasing function, or
\begin{equation}
\label{nonsymmetric}
a_{ij}(t;s)=\frac{N\psi(|x_j(s)-x_i(t)|)}{\sum_{k=1}^N \psi(|x_k(s)-x_i(t)|)}, \quad \forall \ t\geq 0.
\end{equation}
Without loss of generality, we can assume that $\psi (0)=1$. We notice that in both cases we have that
\begin{equation}
\label{condizione1}
\frac{1}{N}\sum_{j=1}^N a_{ij}(t;s)\leq 1, \quad \forall \ t\geq 0.
\end{equation}
Moreover, $\alpha :[0,\tau(0)]\rightarrow [0,+\infty)$ is a weight function which satisfies
\begin{equation*}
\underline{A} :=\int_0^{\tau_*} \alpha(s) ds >0.
\end{equation*}
Furthermore, we define for any $t\geq 0$
\begin{equation}
\label{h(t)}
h(t):= \int_0^{\tau(t)} \alpha (s) ds.
\end{equation}
\begin{Remark}
We notice that if $\alpha(s)=\delta_{\tau(t)} (s)$, then system \eqref{modello} can be rewritten as
$$
\begin{array}{l}
\displaystyle{ \frac{dx_i(t)}{dt}= \frac{1}{N}\sum_{j\neq i} a_{ij}(t;t-\tau(t))(x_j(t-\tau(t))-x_i(t)),}\\
\displaystyle{ x_i(s)=x_{i,0}(s), \quad s\in [-\tau(0),0],}
\end{array}
$$
which is already analyzed in \cite{CPP}.
\end{Remark}
We define, now, the following quantity:
$$
d_X(t):=\max_{1\leq i,j\leq N} |x_i(t)-x_j(t)|.
$$
\begin{Definition}
We say that a solution $\{ x_i(t)\}_{i=1,\dots , N}$ to \eqref{modello} converges to \emph{consensus} if 
$$
\lim_{t\rightarrow +\infty } d_X(t)=0.
$$
\end{Definition}
We will prove the following consensus result.
\begin{Theorem}\label{Teorema_1}
Let $\{ x_i(t)\} _{i=1}^N $ be the solution to \eqref{modello}. Suppose that 
\begin{equation}
\label{condizione_ritardo}
\left( e^{\tau(0)}-1\right) h(0) \leq \frac{\underline{A}\psi(2R)^3}{2+\psi(2R)^2}. 
\end{equation}
Then, there exist two positive constants $C,K$ such that
\begin{equation}
\label{tesi_teorema1}
d_X(t)\leq Ce^{-Kt}, \qquad \forall \ t\geq 0.
\end{equation}
\end{Theorem}
\begin{Remark}
Here, we stress the fact that the quantity
$$
\left( e^{\tau(0)}-1\right) \int_0^{\tau(0)} \alpha(s) ds
$$ 
is increasing with respect to $\tau(0)$.  Then, \eqref{condizione_ritardo} represents a smallness assumption on $\tau(0)$. Moreover, the right-hand side of \eqref{condizione_ritardo} is increasing with respect to $\psi(2R)$. Therefore, we observe that if $R$ is small enough and/or the decay of $\psi$ is not too fast, then the quantity $$ \frac{\psi(2R)^3}{2+\psi(2R)^2} $$ becomes large and consensus occurs for more values of $\tau(0)$.  
\end{Remark}
The transport equation associated to \eqref{modello} can be obtained as mean-field limit of the particle system \eqref{modello} when $N\rightarrow +\infty$. Let $\mathcal{M}(\RR^d)$ be the set of probability measures on the space $\RR^d$. Then, the transport equation associated to \eqref{modello} reads as
\begin{equation}
\label{kinetic}
\begin{array}{l}
\displaystyle{\partial_t \mu_t+\dv \left( \frac{1}{h(t)}\int_{t-\tau(t)}^t \alpha (t-s) F[\mu_s] ds\ \mu_t \right) =0,\quad x\in \RR^d, \quad t\geq 0}\\
\displaystyle{ \mu_s=g_s \quad s\in[-\tau(0),0],}
\end{array}
\end{equation}
where $F$ is given by either
\begin{equation}
\label{F_symm}
F[\mu_s](x)=\int_{\RR^d} \psi(|x-y|)(y-x) d\mu_s(y), 
\end{equation} 
or
\begin{equation}
\label{F_nonsymm}
F[\mu_s](x)=\frac{\int_{\RR^d} \psi (|x-y|)(y-x) d\mu_s(y)}{\int_{\RR^d} \psi(|x-y|)d\mu_s(y)},
\end{equation}
according to the choice of \eqref{symmetric} and \eqref{nonsymmetric}. Furthermore, we take $g_s\in \mathcal{C}([-\tau(0),0];\mathcal{M}(\RR^d))$. 
\begin{Definition}
Let $T>0$. We say that $\mu_t\in \mathcal{C}([0,T);\mathcal{M}(\RR^d))$ is a weak solution to \eqref{kinetic} on the time interval $[0,T)$ if for all $\varphi\in \mathcal{C}^\infty_c(\RR^d\times [0,T))$ we have the following result:
\begin{equation}
\label{weak_solution}
\int_0^T \int_{\RR^d} \left( \partial_t\varphi+\frac{1}{h(t)}\int_{t-\tau(t)}^t \alpha(t-s)F[\mu_s](x)ds \cdot \nabla_x\varphi \right) d\mu_t(x) dt +\int_{\RR^d} \varphi (x,0)dg_0(x) =0,
\end{equation}
where $F[\mu_s]$ is defined as in \eqref{F_symm} or \eqref{F_nonsymm}.
\end{Definition}
We will prove the following theorem.
\begin{Theorem}\label{Teorema_2}
Let $\mu_t\in \mathcal{C}([0,T];\mathcal{P}_1(\RR^d))$ be a weak solution to \eqref{kinetic}, with compactly supported initial measure $g_s\in \mathcal{C}([-\tau(0),0];\mathcal{P}_1(\RR^d))$ and let $F$ as in \eqref{F_symm} or \eqref{F_nonsymm}. Suppose that
\begin{equation}
\label{delaycondition}
\left( e^{\tau(0)}-1\right) h(0)\leq \frac{\underline{A}\psi(2R)^3}{2+\psi(2R)^2}.
\end{equation}
Then, there exists a constant $C>0$ independent of $t$ such that
\begin{equation}
\label{tesi finale}
d_X(\mu_t)\leq \left( \max_{s\in[-\tau(0),0]} d_X(g_s) \right) e^{-Ct},
\end{equation}
for all $t\geq 0$, where $$
d_X(\mu_t):=\text{diam} \ supp \ \mu_t. 
$$
\end{Theorem}
The paper is organized as follows. In Section \ref{particle_part} we study the consensus behavior of solution to \eqref{modello}, after assuming an upper-bound on the initial delay $\tau(0)$, namely we will prove Theorem \ref{Teorema_1}. In Section \ref{kinetic_part} we focus our attention on system \eqref{kinetic} and we study the existence and uniqueness of the solution and its convergence to consensus.
\section{Consensus results}\label{particle_part}
We notice that $d_X$ may be not differentiable at some $t\geq 0$. Then, we will use a suitable generalized derivative. We define the upper Dini derivative of a continuous function $F$ as follows:
$$
D^+ F(t):=\limsup_{h\rightarrow 0^+} \frac{F(t+h)-F(t)}{h}.
$$

Before studying the convergence to consensus of the solution to \eqref{modello}, we state the following lemma.
\begin{Lemma}\label{lemma1}
Let $\{ x_i(t)\} _{i=1}^N$ be a solution to \eqref{modello}. Suppose that the initial functions $x_{i,0}(s)$ are continuous on the time interval $[-\tau(0),0]$ for all $i=1,\dots,N$. Set $$R:=\max_{s\in [-\tau(0),0]}\max_{1\leq i \leq N} |x_i(s)|.$$ Then,
\begin{equation}\label{TESI1}
\max_{1\leq i\leq N} |x_i(t)|\leq R
\end{equation}
for all $t\geq 0$.
\end{Lemma}
\begin{proof}
Let $\epsilon>0$ and define $R_\epsilon := R+\epsilon$. Set
$$
S^\epsilon =\left\{ t>0 \ : \ \max_{1\leq i\leq N} |x_i(s)|< R_\epsilon , \quad \forall \ s \in [0,t) \right\} . 
$$
By continuity, $S^\epsilon \neq \emptyset$. Denote $T^\epsilon := \sup S^\epsilon$ and assume by contradiction that $T^\epsilon <+\infty$. Then,
\begin{equation}
\label{contraddizione}
\lim_{t\rightarrow T^{\epsilon -}} \max _{1\leq i\leq N} |x_i(t)|=R^\epsilon.
\end{equation}

On the other hand, we have that for any $t\leq T^\epsilon$,
$$
\begin{array}{l}
\displaystyle{ \frac{1}{2}D^+|x_i(t)|^2\leq\left\langle x_i(t),\frac{dx_i(t)}{dt}\right\rangle }\\
\hspace{1.98 cm}
\displaystyle{ =\left\langle x_i(t), \frac{1}{Nh(t)} \sum_{j\neq i} \int_{t-\tau(t)}^t \alpha (t-s) a_{ij}(t;s) (x_j(s)-x_i(t)) ds \right\rangle }\\
\hspace{1.98 cm}
\displaystyle{ =\frac{1}{Nh(t)}\sum_{j\neq i} \int_{t-\tau(t)}^t \alpha (t-s)a_{ij}(t;s)\langle x_i(t),x_j(s)-x_i(t)\rangle ds}\\
\hspace{1.98 cm}
\displaystyle{ =\frac{1}{Nh(t)}\sum_{j\neq i}\int_{t-\tau(t)}^t \alpha (t-s) a_{ij}(t;s) \left( \langle x_i(t),x_j(s)\rangle -|x_i(t)|^2\right) ds}\\
\hspace{1.98 cm}
\displaystyle{ \leq \frac{1}{Nh(t)}\sum_{j\neq i} \int_{t-\tau(t)}^t \alpha(t-s) a_{ij}(t;s) |x_i(t)|\left( |x_j(s)|-|x_i(t)|\right) ds.} 
\end{array}
$$
Using \eqref{condizione1} and the fact that $t\leq T^\epsilon$ yield
$$
\begin{array}{l}
\displaystyle{ \frac{1}{2}D^+|x_i(t)|^2\leq \frac{1}{h(t)}\int_{t-\tau(t)}^t \alpha(t-s)ds\ |x_i(t)| (R^\epsilon-|x_i(t)|)=|x_i(t)|(R^\epsilon-|x_i(t)|).}
\end{array}
$$
Hence, we have that
$$
D^+|x_i(t)|\leq R^\epsilon-|x_i(t)|.
$$
By Gronwall inequality, we obtain
$$
|x_i(t)|\leq e^{-t} \left( |x_i(0)|-R^\epsilon \right)+R^\epsilon < R^\epsilon.
$$
Therefore, 
$$
\lim_{t\rightarrow T^{\epsilon -}} \max_{1\leq i\leq N} |x_i(t)| < R^\epsilon,
$$
which is in contradiction with \eqref{contraddizione}. Moreover, since $\epsilon$ is arbitary, we obtain \eqref{TESI1}.
\end{proof}
\begin{Remark}
Thanks to the previous lemma, we can find a control on $a_{ij}(t;s)$ from below. Indeed, for any $i,j\in \{1,\dots,N\}$, for any $t\geq 0$ and $s\in[t-\tau(t),t]$, we have that
$$
|x_j(s)-x_i(t)|\leq |x_j(s)|+|x_i(t)|\leq 2R.
$$
Hence, from \eqref{symmetric} and \eqref{nonsymmetric}, we can deduce that
\begin{equation}
\label{stima_dal_basso}
a_{ij}(t;s)\geq \psi (2R), \quad \forall\ t\geq 0.
\end{equation}
\end{Remark}
\begin{Lemma}\label{Lemma2}
Let $\{ x_i(t)\}_{i=1}^N$ be the solution to \eqref{modello}. Moreover, define
\begin{equation}
\label{gamma(t)}
\gamma (t):=\frac{1}{h(t)}\int_{t-\tau(t)}^t \alpha (t-s)\int_s^t \max_{1\leq k\leq N} \left| \frac{dx_k(z)}{dz}\right| dz ds, \quad \forall\ t\geq 0.
\end{equation}
Then, 
\begin{equation}\label{tesi_lemma2}
D^+ d_X(t) \leq \frac{2}{\psi(2R)}\gamma(t)-\psi(2R)d_X(t), \quad \forall\ t\geq 0.
\end{equation}
\end{Lemma}
\begin{proof}
Due to continuity of $x_i(t)$, for any $i\in \{1,\dots, N\}$, there exists a sequence of times $\{ t_k\}_{k\in \nat}$ such that 
$$
\bigcup _{k\in \nat} [t_k, t_{k+1}) =[0,+\infty),
$$
and for each $k\in \nat$ and for any $ t\in (t_k,t_{k+1})$ there exist $i,j\in \{ 1,\dots,N\}$ such that
$$
d_X(t)= |x_i(t)-x_j(t)|.
$$
Hence, we have that
\begin{equation}\label{Q1}
\begin{array}{l}
\displaystyle{
\frac{1}{2}D^+ d_X^2(t)\leq \left\langle x_i(t)-x_j(t), \frac{dx_i(t)}{dt}-\frac{dx_j(t)}{dt}\right\rangle }\\
\hspace{1.5 cm}\displaystyle{=\frac{1}{Nh(t)}\left\langle x_i(t)-x_j(t), \sum_{k\neq i} \int_{t-\tau(t)}^t \alpha(t-s) a_{ik}(t;s)(x_k(s)-x_i(t)) ds \right\rangle}\\
\hspace{2 cm}\displaystyle{ - \frac{1}{Nh(t)}\left\langle x_i(t)-x_j(t), \sum_{k\neq j}\int_{t-\tau(t)}^t \alpha(t-s)a_{jk}(t;s)(x_k(s)-x_j(t))ds \right\rangle}\\
\hspace{1.5 cm}\displaystyle{ =:I_1+I_2.}
\end{array}
\end{equation}
Now, $I_1$ and $I_2$ can be rewritten in the following way:
\begin{equation}
\label{passaggio1}
\begin{array}{l}
\displaystyle{ I_1=\frac{1}{Nh(t)}\sum_{k\neq i}\int_{t-\tau(t)}^t \alpha (t-s)  a_{ik}(t;s)\langle x_i(t)-x_j(t), x_k(s)-x_k(t)\rangle ds}\\
\hspace{1.5 cm}
\displaystyle{ + \frac{1}{Nh(t)}\sum_{k\neq i} \int_{t-\tau(t)}^t \alpha (t-s) a_{ik}(t;s)\langle x_i(t)-x_j(t),x_k(t)-x_i(t)\rangle ds}
\end{array}
\end{equation}
and
$$
\begin{array}{l}
\displaystyle{ I_2=-\frac{1}{Nh(t)}\sum_{k\neq j}\int_{t-\tau(t)}^t \alpha (t-s) a_{jk}(t;s) \langle x_i(t)-x_j(t),x_k(s)-x_k(t)\rangle ds}\\
\hspace{1.5 cm}
\displaystyle{ - \frac{1}{Nh(t)}\sum_{k\neq j}\int_{t-\tau(t)}^t \alpha (t-s) a_{jk}(t;s) \langle x_i(t)-x_j(t),x_k(t)-x_j(t)\rangle ds.}
\end{array}
$$
We observe (as in \cite{CPP}) that for any $t\geq 0$,
$$
\langle x_i(t)-x_j(t), x_k(t)-x_i(t)\rangle \leq 0, \quad \forall k\in \{ 1,\dots , N \}.
$$
Moreover, we notice that for any $i,j \in \{ 1,\dots , N\}$ 
\begin{equation}\label{stima_alto}
a_{ij}(t;s)\leq \frac{1}{\psi(2R)}
\end{equation}
in both cases \eqref{symmetric} and \eqref{nonsymmetric}. Indeed, if $a_{ij}$ are as in \eqref{nonsymmetric}, for any $i,j=1,...,N$, then we obtain \eqref{stima_alto}, using \eqref{stima_dal_basso} and the fact that $\psi$ is a non-increasing function with $\psi(0)=1$. Moreover, if we take $a_{ij}$ as in \eqref{symmetric}, then \eqref{stima_alto} immediately follows, using the fact that $a_{ij}(t;s)\leq 1$, for any $i,j=1,...,N$, and $\psi(2R)\leq 1$. Therefore, using \eqref{stima_dal_basso} and \eqref{stima_alto} in \eqref{passaggio1} yield
\begin{equation}
\label{stima2}
\begin{array}{l}
\displaystyle{ I_1 \leq \frac{1}{Nh(t)} \frac{d_X(t)}{\psi(2R)} \sum_{k\neq i} \int_{t-\tau(t)}^t \alpha (t-s)  |x_k(s)-x_k(t)| ds}\\
\hspace{3.5 cm}
\displaystyle{+\frac{\psi(2R)}{N}\sum_{k=1}^N \langle x_i(t)-x_j(t), x_k(t)-x_i(t)\rangle  .}
\end{array}
\end{equation} 
As before, we observe that for any $t\geq 0$
$$
-\langle x_i(t)-x_j(t), x_k(t)-x_j(t)\rangle \leq 0, \quad \forall k\in \{ 1,\dots ,N\} .
$$
Hence, using again \eqref{stima_dal_basso} and \eqref{stima_alto}, we can obtain a similar estimate for $I_2$, namely
\begin{equation}\label{stima_I2}
\begin{array}{l}
\displaystyle{ I_2 \leq \frac{1}{Nh(t)}\frac{d_X(t)}{\psi(2R)} \sum_{k\neq j} \int_{t-\tau(t)}^t \alpha (t-s) |x_k(s)-x_k(t)| ds}\\
\hspace{3,5 cm}
\displaystyle{ +\frac{\psi(2R)}{N}\sum_{k=1}^N \langle x_i(t)-x_j(t), x_j(t)-x_k(t)\rangle .}
\end{array} 
\end{equation}
Using \eqref{stima2} and \eqref{stima_I2} in \eqref{Q1}, we have that
\begin{equation}
\label{Q2}
\begin{array}{l}
\displaystyle{ \frac{1}{2}D^+ d_X(t)^2 \leq \frac{2}{Nh(t)}\frac{d_X(t)}{\psi(2R)}\sum_{k=1}^N \int_{t-\tau(t)}^t \alpha (t-s) |x_k(s)-x_k(t)| ds-\psi(2R) d_X(t)^2.}
\end{array}
\end{equation}
Moreover, we notice that, for $s<t$,
$$
\sum_{k=1}^N |x_k(s)-x_k(t)| \leq \sum_{k=1}^N \int_s^t \left| \frac{dx_k(z)}{dz}\right| dz \leq N \int_s^t \max_{1\leq k\leq N} \left| \frac{dx_k(z)}{dz}\right| dz.
$$
Substituting this estimate in \eqref{Q2}, we obtain 
$$
\begin{array}{l}
\displaystyle{ \frac{1}{2}D^+ d_X(t)^2 \leq \frac{2d_X(t)}{\psi(2R)}\gamma(t)-\psi(2R) d_X(t)^2 ,}
\end{array}
$$
which yields \eqref{tesi_lemma2}.
\end{proof}
\begin{Lemma}\label{Lemma3}
Let $\{ x_i(t)\}_{i=1}^N$ be the solution to \eqref{modello}. Then, for any $t\geq 0$ 
\begin{equation}
\label{tesi_lemma3}
\max_{1\leq i\leq N} \left| \frac{dx_i(t)}{dt}\right| \leq \frac{1}{\psi(2R)}\gamma(t)+\frac{1}{\psi(2R)}d_X(t).
\end{equation}
\end{Lemma}
\begin{proof}
We have that for any $i\in \{ 1,\dots , N\}$, 
$$
\begin{array}{l}
\displaystyle{ \left| \frac{dx_i(t)}{dt}\right| \leq \frac{1}{Nh(t)}\sum_{k\neq i} \int_{t-\tau(t)}^t \alpha(t-s) a_{ik}(t;s) |x_k(s)-x_k(t)| ds}\\
\hspace{2 cm}
\displaystyle{+ \frac{1}{Nh(t)}\sum_{k\neq i} \int_{t-\tau(t)}^t \alpha (t-s) a_{ik}(t;s) |x_k(t)-x_i(t)|ds.}
\end{array}
$$
Using \eqref{stima_alto} yields
$$
\begin{array}{l}
\displaystyle{ \left| \frac{dx_i(t)}{dt}\right| \leq \frac{1}{\psi(2R)}\gamma(t)+\frac{1}{\psi(2R)} d_X(t).} 
\end{array}
$$
Taking the maximum, we obtain \eqref{tesi_lemma3}.
\end{proof}
Now, we are ready to prove Theorem \ref{Teorema_1}.

\begin{proof}[Proof of Theorem \ref{Teorema_1}]
Define the following Lyapunov functional:
$$
\mathcal{L}(t) := d_X(t)+\beta \int_0^{\tau(t)} \alpha (s) \int_{t-s}^t e^{-(t-\sigma)} \int_{\sigma}^t \max_{1\leq k\leq N} \left| \frac{dx_k(\rho)}{d\rho}\right| d\rho d\sigma ds,
$$
with $\beta >0$. Then,
$$
\begin{array}{l}
\displaystyle{ D^+\mathcal{L}(t)= D^+ d_X(t)+\beta \tau'(t) \alpha (\tau(t))\int_{t-\tau(t)}^t e^{-(t-\sigma)} \int_\sigma^t \max_{1\leq k\leq N} \left| \frac{dx_k(\rho)}{d\rho}\right| d\rho d\sigma }\\
\hspace{2 cm}
\displaystyle{ -\beta \int_0^{\tau(t)}  \alpha(s)e^{-s} \int_{t-s}^t \max_{1\leq k\leq N} \left| \frac{dx_k(\rho)}{d\rho}\right| d\rho ds}\\
\hspace{2 cm}
\displaystyle{-\beta  \int_0^{\tau(t)} \alpha (s) \int_{t-s}^t e^{-(t-\sigma)} \int_{\sigma}^t \max_{1\leq k\leq N} \left| \frac{dx_k(\rho)}{d\rho}\right| d\rho d\sigma ds}\\
\hspace{2 cm}
\displaystyle{ + \beta \max_{1\leq k\leq N} \left| \frac{dx_k(t)}{dt}\right| \int_0^{\tau(t)} \alpha(s) \int_{t-s}^t e^{-(t-\sigma)} d\sigma ds.}

\end{array}
$$
Using $\underline{A}\leq h(t) \leq h(0)$ and $\tau'(t)\leq0$, we deduce
$$
\begin{array}{l}
\displaystyle{ D^+\mathcal{L}(t)\leq D^+ d_X(t) -\beta e^{-\tau(0)} \underline{A} \gamma(t)-\beta \int_0^{\tau(t)} \alpha (s) \int_{t-s}^t e^{-(t-\sigma)} \int_{\sigma}^t \max_{1\leq k\leq N} \left| \frac{dx_k(\rho)}{d\rho}\right| d\rho d\sigma ds}\\
\hspace{2.5 cm}
\displaystyle{ +\beta h(0)(1-e^{-\tau(0)})\max_{1\leq k\leq N} \left| \frac{dx_k(t)}{dt}\right| .}
\end{array}
$$
Now, since \eqref{tesi_lemma2} and \eqref{tesi_lemma3} hold, we have that
$$
\begin{array}{l}
\displaystyle{ D^+\mathcal{L}(t)\leq \left( \frac{2}{\psi(2R)}-\beta e^{-\tau(0)} \underline{A}+\beta h(0)(1-e^{-\tau(0)})\frac{1}{\psi(2R)} \right) \gamma(t)}\\
\displaystyle{ + \left( -\psi(2R)+\beta h(0) \frac{1-e^{-\tau(0)}}{\psi(2R)} \right) d_X(t) -\beta \int_0^{\tau(t)} \alpha (s) \int_{t-s}^t e^{-(t-\sigma)} \int_{\sigma}^t \max_{1\leq k\leq N} \left| \frac{dx_k(\rho)}{d\rho}\right| d\rho d\sigma ds.}
\end{array}
$$
We want to show that for $\tau(0)$ sufficiently small we obtain the existence of $K>0$ such that
\begin{equation}
\label{scopo_teorema}
D^+ \mathcal{L}(t)\leq -K\mathcal{L}(t), \quad \forall \ t\geq 0.
\end{equation}
This is true if the following two conditions hold:
\begin{equation}
\displaystyle{\frac{2}{\psi(2R)}-\beta e^{-\tau(0)} \underline{A}+\beta h(0)(1-e^{-\tau(0)})\frac{1}{\psi(2R)}\leq  0, \label{cond1}}
\end{equation}
\begin{equation}
\displaystyle{ -\psi(2R)+\beta h(0) \frac{1-e^{-\tau(0)}}{\psi(2R)}< 0. \label{cond2}}
\end{equation}
The inequality \eqref{cond2} is satisfied for
\begin{equation}
\label{beta1}
\beta < \frac{\psi(2R)^2}{h(0)(1-e^{-\tau(0)})}.
\end{equation} 
Now, in order to have \eqref{cond1}, we need
$$
h(0) \left( e^{\tau(0)}-1\right) < \underline{A}\psi(2R).
$$
Hence, \eqref{cond1} is satisfied if
\begin{equation}
\label{beta2}
\beta \geq \frac{2}{e^{-\tau(0)} \underline{A}\psi(2R)-h(0)(1-e^{-\tau(0)})}.
\end{equation}
Then, in order to have the existence of the parameter $\beta>0$ such that \eqref{beta1} and \eqref{beta2} hold, we need
$$
\frac{2}{e^{-\tau(0)} \underline{A}\psi(2R)-h(0)(1-e^{-\tau(0)})} < \frac{\psi(2R)^2}{h(0)(1-e^{-\tau(0)})},
$$
which is true for any $\tau(0)$ satisfying \eqref{condizione_ritardo}. Choosing
$$
K=\min \left\{ \beta, \psi(2R)-\beta h(0) \frac{1-e^{-\tau(0)}}{\psi(2R)} \right\} ,
$$ 
we obtain \eqref{scopo_teorema}. We notice that since $\beta$ satisfies \eqref{beta1}, then $K>0$. This implies immediately \eqref{tesi_teorema1}. Hence, the theorem is proved.
\end{proof}
\section{Consensus of solution to \eqref{kinetic}}\label{kinetic_part}
In this section we want to analyse the transport equation \eqref{kinetic} associated to \eqref{modello}, obtained as mean-field limit of the particle system when $N\rightarrow +\infty$. To do so, we consider $\psi$ Lipschitz continuous and we denote by $L$ its Lipschitz constant. 

Before proving the existence and uniqueness of solutions to \eqref{kinetic}, we first recall some tools on probability spaces and measures.
\begin{Definition}
Let $\mu,\nu\in \mathcal{M}(\RR^d)$ be two probability measures on $\RR^d$. We define the 1-Wasserstein distance between $\mu$ and $\nu$ as
$$
d_1(\mu,\nu):=\inf_{\pi\in \Pi(\mu,\nu)} \int_{\RR^d\times \RR^d} |x-y|d\pi(x,y),
$$
where $\Pi (\mu,\nu)$ is the space of all couplings for $\mu$ and $\nu$, namely all those probability measures on $\RR^{2d}$ having as marginals $\mu$ and $\nu$:
$$
\int_{\RR^d\times\RR^d} \varphi (x)d\pi(x,y)=\int_{\RR^d} \varphi(x)d\mu(x), \quad \int_{\RR^d\times\RR^d} \varphi(y)d\pi(x,y)=\int_{\RR^d}\varphi(y)d\nu(y),
$$
for all $\varphi \in \mathcal{C}_b(\RR^d)$.
\end{Definition}
It's well-known that $(\mathcal{P}_1(\RR^d),d_1)$ (where $\mathcal{P}_1$ is the space of all probability measures with finite first-order moment) is a complete metric space. Moreover, in order to prove the existence of solution to \eqref{kinetic}, we need the following definition.
\begin{Definition}
Let $\mu$ be a Borel measure on $\RR^d$ and let $T:\RR^d\rightarrow\RR^d$ be a measurable map. We define the push-forward of $\mu$ via $T$ as the measure
$$
T\# \mu(A):=\mu(T^{-1}(A)),
$$
for all Borel sets $A\subset \RR^d$.
\end{Definition}
Then, we have the following theorem.
\begin{Theorem}
\label{esistenza_unicita}
Consider the system \eqref{kinetic} with $g_s\in \mathcal{C}([-\tau(0),0];\mathcal{P}_1(\RR^d)).$ Suppose that there exists a constant $R>0$ such that
$$
supp \ g_t \in B^d(0,R),
$$
for all $t\in [-\tau(0),0],$ where $B^d(0,R)$ denotes the ball of radius $R$ in $\RR^d$ centered at the origin. Then, for any $T>0$ there exists a unique weak solution $\mu_t\in \mathcal{C}([0,T);\mathcal{P}_1(\RR^d))$ of \eqref{kinetic} in the sense of \eqref{weak_solution}. Moreover, $\mu_t$ is uniformly compactly supported and \begin{equation}
\label{pushforward}
\mu_t =X(t;\cdot)\#\mu_0,
\end{equation}
where $X(t;\cdot)$ is the solution of the characteristic system associated to \eqref{kinetic} for any $t\in[0,T)$.
\end{Theorem}
\begin{proof}
First of all we claim that for any $t\in[0,T]$, there exist two positive constants $C,K>0$ such that 
$$
\left| \frac{1}{h(t)}\int_{t-\tau(t)}^t\alpha(t-s) F[\mu_s](x)ds-\frac{1}{h(t)}\int_{t-\tau(t)}^t \alpha(t-s) F[\mu_s](\tilde{x}) ds\right| \leq C|x-\tilde{x}|,
$$
for any $x,\tilde{x} \in B^d(0,R)$, and
$$
\left| \frac{1}{h(t)}\int_{t-\tau(t)}^t \alpha(t-s) F[\mu_s](x)ds\right| \leq K,
$$
for all $x\in B^d(0,R)$, with $F$ as in \eqref{F_symm} or in \eqref{F_nonsymm}. The proof of this claim is very similar to \cite[Lemma 3.4]{CPP}. Then, from \cite[Theorem 3.10]{CCR}, we deduce that there exists a unique weak solution to \eqref{kinetic} in the sense of \eqref{weak_solution} and it exists as long as $\mu_t$ is compactly supported. Hence, we need to estimate the growth of support. To do so, we set
$$
R_X[\mu_t]:=\max_{x\in \overline{supp\ \mu_t}} |x|,
$$
for $t\in[0,T]$ and we define
$$
R_X(t):=\max_{-\tau(0)\leq s\leq t} R_X[\mu_s].
$$
Now, we proceed by steps. We consider $t\in[0,\tau_*]$ and we construct the system of characteristics $X(t;x):[0,\tau_*]\times \RR^d\rightarrow \RR^d$ associated to \eqref{kinetic}:
\begin{equation}
\label{caratteristiche}
\begin{array}{l}
\displaystyle{\frac{dX(t;x)}{dt}=\frac{1}{h(t)}\int_{t-\tau(t)}^t \alpha (t-s) F[\mu_s](X(s;x)) ds,}\\
\displaystyle{ X(0;x)=x,\quad x\in \RR^d.}
\end{array}
\end{equation} 
We notice that the system \eqref{caratteristiche} is well-defined, since the velocity field
$$
\frac{1}{h(t)}\int_{t-\tau(t)}^t \alpha (t-s) F[\mu_s] ds
$$
is locally Lipschitz and locally bounded. Then, arguing as in Lemma \ref{lemma1}, we have that
$$
\frac{d|X(t;x)|}{dt}\leq R_X(t)-|X(t;x)|,
$$ 
which yields
$$
R_X(t)<R_X(0),
$$
for any $t\in[0,\tau_*]$. Thus, we obtain a unique solution $\mu_t$ to \eqref{kinetic} on the time interval $[0,\tau_*]$. We can iterate this process on all the intervals of the type $[k\tau_*,(k+1)\tau_*]$, with $k=1,2,\dots$, until we reach the final time $T$. Moreover, following \cite{CCR}, it's possible to find a measure $\mu_t$ which satisfies \eqref{pushforward} and this is equivalent to the definiton of weak solution \eqref{weak_solution}.
\end{proof}
\subsection{Consensus behavior}
In this subsection we will prove the consensus behavior of the solution to \eqref{kinetic}, with $F$ as in \eqref{F_symm} or \eqref{F_nonsymm}. To do so, we firstly need the following stability result.
\begin{Lemma}
\label{stability}
Let $\mu^1_t,\mu_t^2\in \mathcal{C}([0,T];\mathcal{P}_1(\RR^d))$ be two weak solutions to \eqref{kinetic}, with compactly supported initial data $g^1_s,g^2_s\in \mathcal{C}([-\tau(0),0];\mathcal{P}_1(\RR^d))$ respectively. Then, there exists a constant $C>0$ depending only on $T$ such that
\begin{equation}
\label{stability_result}
d_1(\mu^1_t,\mu^2_t)\leq C\max_{s\in[-\tau(0),0]} d_1(g^1_s,g^2_s),
\end{equation}
for any $t\in[0,T]$.
\end{Lemma}
\begin{proof}
For $i=1,2$ let $X^i(t;x):[0,T]\times \RR^d\rightarrow \RR^d$ be the characteristics associated to \eqref{kinetic}, which obey to 
$$
\begin{array}{l}
\displaystyle{\frac{dX^i(t;x)}{dt}=\frac{1}{h(t)}\int_{t-\tau(t)}^t \alpha (t-s) F[\mu_s](X^i(s;x)) ds,}\\
\displaystyle{ X^i(0;x)=x,}
\end{array}
$$
for any $ x \in \RR^d $. We remember that the characteristics $ X^i $ are well-defined in $[0,T]$ since, by Theorem \ref{esistenza_unicita}, $ \mu^i_t $ have uniformly compact support on such interval. Then, we have that
$$
\mu^i_t =X^i(t;\cdot) \# \mu^i_s, \quad \forall t,s\in [0,T].
$$
Moreover, as before, we define
$$
R^T_{i;X}:=\max_{s\in[-\tau(0),T]} R_X[\mu^i_s].
$$
Then, we choose an optimal transport map between $\mu^1_0$ and $\mu^2_0$ with respect to $d_1$ (call it $S_0(x)$) such that $\mu^2_0=S_0\# \mu^1_0$ and
$$
d_1(\mu^1_0,\mu^2_0)=\int_{\RR^d} |x-S_0(x)|d\mu^1_0(x).
$$
Moreover, we define the map $T^t$ for any $t\in [0,T]$ as
\begin{equation}
\label{optimal}
T^t:=X^2(t;\cdot) \circ S_0\circ X^1(t;\cdot)^{-1}.
\end{equation}
Therefore, we can write
$$
T^t \# \mu^1_t=\mu^2_t, \quad \forall t\in[0,T]
$$
and 
$$
d_1(\mu^1_t,\mu^2_t)\leq \int_{\RR^d} |x-T^t(x)|d\mu^1_t(x):=u(t).
$$
Using \eqref{optimal} yields
$$
u(t)=\int_{\RR^d} \left| X^1(t;x)-X^2(t;S_0(x))\right| d\mu^1_0(x).
$$
Moreover, we extend the definition of $T^t$ on the interval $[-\tau(0),0]$ and we define $u(t)$ for $t\in [-\tau(0),0]$ as
$$
u(t):=d_1(g^1_t,g^2_t)=\int_{\RR^d} \left| x-T^t(x)\right| dg^1_t(x).
$$
Now, differentiating $u(t)$ and using \eqref{optimal}, we obtain
$$
\frac{du(t)}{dt}\leq \frac{1}{h(t)}\int_{\RR^d}\int_{t-\tau(t)}^t \alpha(t-s) \left| F[\mu^1_s](x)-F[\mu^2_s](T^t(x))\right| ds d\mu_t^1(x)=:J.
$$
We consider, now, the case of $F$ as in \eqref{F_symm}. Then,
$$
\begin{array}{l}
\displaystyle{|F[\mu^1_s](x)-F[\mu^2_s](T^t(x))|}\\
\hspace{2 cm}
\displaystyle{\leq \int_{\RR^d}\left| \psi(|x-y|)(y-x)-\psi(|T^t(x)-T^s(y)|)(T^s(y)-T^t(x))\right| d\mu^1_s(y)}\\
\hspace{2 cm}
\displaystyle{ \leq \int_{\RR^d} \left| \psi(|x-y|)-\psi(|T^t(x)-T^s(y)|)\right|\cdot |y-x| d\mu^1_s(y)}\\
\hspace{5.5 cm}
\displaystyle{+ \int_{\RR^d} \psi(|T^t(x)-T^s(y)|) \cdot \left| y-x -(T^s(y)-T^t(x))\right| d\mu^1_s(y)}\\
\hspace{2 cm}
\displaystyle{ =(1)\ +\ (2).}
\end{array}
$$
Now, 
$$
\begin{array}{l}
\displaystyle{ (1)\leq L\int_{\RR^d} \left| x-y-T^t(x)+T^s(y)\right| \cdot |y-x| d\mu^1_s(y)}\\
\hspace{0.55 cm}
\displaystyle{ \leq L(|x|+R_{1;X}^T) \left[ |x-T^t(x)|+\int_{\RR^d} |y-T^s(y)|d\mu_s^1(y)\right],}
\end{array}
$$
and
$$
\begin{array}{l}
\displaystyle{ (2)\leq |x-T^t(x)|+\int_{\RR^d} |y-T^s(y)|d\mu_s^1(y).}
\end{array}
$$
Therefore, there exists a constant $C>0$ depending only on $T$ such that
$$
J\leq C\left( u(t)+\frac{1}{h(t)}\int_{t-\tau(t)}^t \alpha (t-s) u(s) ds \right).
$$
Now, if we take $F$ as in \eqref{F_nonsymm}, we have that
$$
\begin{array}{l}
\displaystyle{ \left| F[\mu_s^1](x)-F[\mu_s^2](T^t(x))\right|}\\
\displaystyle{ = \Biggl| \frac{\int_{\RR^d} \psi(|x-y|)(y-x)d\mu_s^1(y)}{\int_{\RR^d} \psi(|x-y|)d\mu_s^1(y)}  
 -\frac{\int_{\RR^d}\psi(|T^t(x)-y|)(y-T^t(x))d\mu_s^2(y)}{\int_{\RR^d} \psi(|T^t(x)-y|)d\mu_s^2(y)}\Biggr| }\\
 \displaystyle{ \leq \frac{1}{\psi(R_{1;X}^T)} \left| \int_{\RR^d} \psi(|x-y|)(y-x)d\mu_s^1(y)-\int_{\RR^d} \psi(|T^t(x)-y|)(y-T^t(x))d\mu_s^2(y)\right|}\\
 \hspace{2 cm}
 \displaystyle{ +\frac{1}{\psi(R_{1;X}^T)\psi(R_{2;X}^T)}\left| \int_{\RR^d} \psi(|T^t(x)-y|)(y-T^t(x))d\mu_s^2(y)\right|}\\
 \hspace{4 cm}
 \displaystyle{ \times \left| \int_{\RR^d} \psi(|x-y|)d\mu_s^1(y) -\int_{\RR^d} \psi(|T^t(x)-y|)d\mu_s^2(y) \right| .}  
\end{array}
$$
As before we have that
$$
\begin{array}{l}
\displaystyle{ \left| \int_{\RR^d} \psi(|x-y|)(y-x)d\mu_s^1(y)-\int_{\RR^d} \psi(|T^t(x)-y|)(y-T^t(x))d\mu_s^2(y)\right|}\\
\hspace{5 cm}
\displaystyle{ \leq \left[ (|x|+R_{1;X}^T)L+1\right] (|x-T^t(x)|+u(s)).}
\end{array}
$$
Furthermore,
$$
\left| \int_{\RR^d} \psi(|T^t(x)-y|)(y-T^t(x))d\mu_s^2(y)\right| \leq R_{2;X}^T + |T^t(x)|,
$$
and
$$
\left| \int_{\RR^d} \psi(|x-y|)d\mu_s^1(y) -\int_{\RR^d} \psi(|T^t(x)-y|)d\mu_s^2(y) \right| \leq L(|x-T^t(x)|+u(s)).
$$
Hence, we obtain again the existence of a constant $C>0$ depending only on $L$ and $T$ such that
$$
\frac{du(t)}{dt}\leq C\left( u(t)+\frac{1}{h(t)}\int_{t-\tau(t)}^t \alpha (t-s) u(s) ds\right).
$$
Denote 
$$
\overline{u}=\max_{s\in[-\tau(0),0]} u(s)=\max_{s\in [-\tau(0),0]} d_1(g^1_s,g^2_s),
$$
and define $w(t):=e^{-Ct}u(t)$. Then, we have that
\begin{equation}
\label{dwdt}
\frac{dw(t)}{dt}\leq \frac{C}{h(t)}\int_{-\tau(0)}^t \alpha (t-s) w(s) ds.
\end{equation}
Thus, we can rewrite \eqref{dwdt} as
$$
\frac{dw(t)}{dt}\leq K\tau(0) \overline{u} +K\int_0^t w(s)ds,
$$
for some $K>0$. This gives us the following estimate:
$$
w(t)\leq \tilde{K}\overline{u}, \qquad \forall t\in [0,T],
$$
for some $\tilde{K}>0$. Then, by definition of $w$ we have
$$
d_1(\mu_t^1,\mu_t^2)\leq u(t)\leq \tilde{K} e^{CT}\overline{u}, \quad \forall \ t\in [0,T],
$$
which gives us the thesis of this lemma.
\end{proof}
We are finally ready to prove Theorem \ref{Teorema_2}.
\begin{proof}[Proof of Theorem \ref{Teorema_2}]
Fixed $g_s\in \mathcal{C}([-\tau(0),0];\mathcal{P}_1(\RR^d))$, we construct the familiy of $N$- particle approximations of $g_s$, which is a family $\{ g_s^N\}_{N\in \nat}$ such that
$$
g_s^N=\sum_{i=1}^N \delta (x-x_i^0(s)),
$$
where $x_i^0 \in \mathcal{C}([-\tau(0),0];\RR^d)$ satisfy
$$
\max_{s\in[-\tau(0),0]} d_1(g_s^N,g_s)\rightarrow 0, \quad \text{as}\ N\rightarrow +\infty.
$$
Moreover, let $\{ x_i^N\}$ be the solution to \eqref{modello}, with initial data $x_i(s)=x_i^0(s)$ for any $s\in[-\tau(0),0]$ and we denote
$$
\mu^N_t:=\sum_{i=1}^N \delta (x-x_i^N(t)),
$$
for any $t\in [0,T]$, which is a weak solution to \eqref{kinetic}. Now, since \eqref{delaycondition} holds, then we know that there exists a constant $C>0$ such that
$$
d_X(t)\leq d_X(0)e^{-Ct}\leq \left( \max_{s\in [-\tau(0),0]} d_X(s)\right) e^{-Ct},
$$
for any $t\geq 0$. Fixing $T\geq 0$, by Lemma \ref{stability} we have that there exists a constant $K>0$ independent of $N$ such that
$$
d_1(\mu_t,\mu_t^N)\leq K \max_{s\in[-\tau(0),0]} d_1(g_s,g_s^N),
$$ 
for any $t\in [0,T]$, where $\mu_t$ is the weak solution to \eqref{kinetic} with initial measure $g_s$. Sending $N\rightarrow +\infty$ we have that $d_X(t)\rightarrow d_X(\mu_t)$ and for any $s\in [-\tau(0),0]$, $d_X(g_s)=d_X(s)$. This gives \eqref{tesi finale} for any $t\in [0,T]$. Since $T$ can be chosen arbitrarly, then the theorem is proved.
\end{proof}
\section{Acknowledgments}
The research of the author is partially supported by the GNAMPA 2019 project \emph{Modelli alle derivate parziali per sistemi multi‐agente} (INdAM).

\end{document}